\documentclass{article}
\usepackage[utf8]{inputenc}
\usepackage{amsthm,amsmath,amssymb,xcolor,placeins}
\usepackage{tikz}
\usepackage{multicol}
\usepackage{enumerate}
\usepackage{xparse}
\usepackage{authblk}
\usepackage{soul}

\usepackage{authblk}

\usepackage{hyperref}
\usetikzlibrary{patterns}

\newtheorem{thm}{Theorem}[section]

\newtheorem{cor}[thm]{Corollary}

\newtheorem{lemma}[thm]{Lemma}

\theoremstyle{definition}

\newtheorem{ex}[thm]{Example}
\newtheorem{obs}[thm]{Observation}

\NewDocumentEnvironment{manual}{O{obs}m}
 {%
 \addtocounter{obs}{-1}%
 \renewcommand\theobs{#2}%
 \begin{#1}
 }
 {\end{#1}}
 
 \NewDocumentEnvironment{manual2}{O{lemma}m}
 {%
 \addtocounter{lemma}{-1}%
 \renewcommand\thelemma{#2}%
 \begin{#1}
 }
 {\end{#1}}

\renewcommand\det{\operatorname{Det}}
\newcommand\dist{\operatorname{Dist}}
\newcommand\aut{\operatorname{Aut}}
\newcommand\tG{\widetilde G}

\newcommand\tT{\widetilde T}
\newcommand\tS{\widetilde S}
\newcommand\tR{\widetilde R}

\newcommand\tB{\widetilde B}
\newcommand\tD{\widetilde D}

\newcommand\gm{\mu_t(G)}

\newcommand\mug{\mu(G)}
\newcommand\Tt{T_t}

\usepackage{todonotes}
\definecolor{spurple}{RGB}{221,160,221}
\definecolor{pgreen}{RGB}{153,255,204}
\definecolor{lblue}{RGB}{186,225,255}
\definecolor{korange}{RGB}{255,204,102}

\title{Determining Number and Cost of Generalized Mycielskian Graphs}

\author[1]{Debra Boutin}
\author[2]{Sally Cockburn}
\author[3]{Lauren Keough}
\author[4]{Sarah Loeb}
\author[5]{K.~E. Perry}
\author[6]{Puck Rombach}
\affil[1]{\url{dboutin@hamilton.edu}, Hamilton College, Clinton, NY}
\affil[2]{\url{scockbur@hamilton.edu}, Hamilton College, Clinton, NY}
\affil[3]{\url{keoulaur@gvsu.edu}, Grand Valley State University, Allendale Charter Township, MI}
\affil[4]{\url{sloeb@hsc.edu}, Hampden-Sydney College, Hampden-Sydney, VA}
\affil[5]{\url{kperry@soka.edu}, Soka University of America, Aliso Viejo, CA}
\affil[6]{\url{puck.rombach@uvm.edu}, University of Vermont, Burlington, VT}

\date{\today}

\begin{document}

\maketitle                      

\begin{abstract}
 A set $S$ of vertices is a \emph{determining set} for a graph $G$ if every automorphism of $G$ is uniquely determined by its action on $S$. The size of a smallest determining set for $G$ is called its \emph{determining number}, $\det(G)$. A graph $G$ is said to be \emph{$d$-distinguishable} if there is a coloring of the vertices with $d$ colors so that only the trivial automorphism preserves the color classes. The smallest $d$ for which $G$ is $d$-distinguishable is its \emph{distinguishing number}, $\dist(G)$. If $\dist(G) = 2$, the \emph{cost of 2-distinguishing}, $\rho(G)$, is the size of a smallest color class over all 2-distinguishing colorings of $G$. The \emph{Mycielskian}, $\mu(G)$, of a graph $G$ is constructed by adding a root vertex $w$, and for each vertex $v_i$ of $G$ adding a shadow vertex $u_i$, with edges so that the neighborhood of $u_i$ in $\mu(G)$ is the same as the neighborhood of $v_i$ in $G$ with the addition of $w$.  The \emph{generalized Mycielskian} $\mu_t(G)$ of a graph $G$ is a Mycielskian graph with $t$ layers of shadow vertices, each with edges to layers above and below, and the root only adjacent to the top layer of shadow vertices. A graph is \emph{twin-free} if  no two vertices have the same neighborhood. This paper examines the determining number and, when relevant, the cost of 2-distinguishing for Mycielskians and generalized Mycielskians of simple graphs with no isolated vertices. In particular, if $G \neq K_2$ is twin-free with no isolated vertices, then $\det(\gm) = \det(G)$. If in addition $\det(G) \geq 2$ and $t\ge \det(G)-1$, then $\dist(\gm)=2$ and $\rho(\gm) = \det(G)$.
 For $G$ with twins, we develop a framework using quotient graphs with respect to equivalence classes of twin vertices to find  $\det(\mu(G))$ and $\det(\mu_t(G))$ in terms of $\det(G)$. 
 
\end{abstract}

\section{Introduction}

In 1955, Mycielski~\cite{M1955} introduced a construction that takes a finite simple graph, $G$, and produces a larger graph, $\mu(G)$, called the (\emph{traditional}) \emph{Mycielskian} of $G$, with a strictly larger chromatic number. To construct $\mu(G)$, begin with a copy of $G$. For each $v\in V(G)$, add a \emph{shadow vertex} $u$ and edges between $u$ and the neighbors of $v$. Finally, add a vertex $w$ whose neighbors are precisely the shadow vertices. We call $w$ the \emph{root}. A more formal definition can be found in Section \ref{sec:autos}.

Mycielski iteratively applied this construction to $G=K_2$, creating what are now called the \emph{classic Mycielski graphs}, $M_n$. More precisely, $M_1 = \mu(K_2)$ and for all $k > 1$, $M_k = \mu(M_{k-1}) = \mu^{k}(K_2)$, where we use $\mu^k(G)$ to indicate iteratively applying the Mycielsian construction $k$ times starting with $G$. Thus, $M_1 = \mu(K_2) = C_5$, and $M_2 = \mu(M_1) = \mu^2(K_2)$, which is commonly called the Gr\"{o}tzsch graph. Mycielski proved that these graphs are all triangle-free and satisfy $\chi(M_n) \ge n$, where $\chi(G)$ is the chromatic number of $G$.

The \emph{generalized Mycielskian }of graph $G$ was defined by Stiebitz~\cite{stiebitz1985beitrage} in 1985 (cited in~\cite{ct2001}) and independently by Van Ngoc~\cite{van1987graph} in 1987 (cited in~\cite{van19954}). It is denoted $\gm$ and will also be formally defined in Section~\ref{sec:autos}. This construction can be described as having a copy of $G$ at level $0$ and $t\geq 1$ levels of shadow vertices whose neighborhoods extend to the levels above and below and that mirror the neighborhoods of the vertices of $G$. Finally, $\gm$ has a root $w$ that is adjacent to each shadow vertex at the top level, $t$. Note that $\mu_1(G) = \mug$. The generalized Mycielski construction is used to construct graphs with arbitrarily large odd girth and arbitrarily large chromatic number. 

One strength of the Mycielskian construction is its ability to build large families of graphs with a given parameter fixed  and other parameters strictly growing.  In the last few decades, this has motivated significant work on parameters of $\mu_t(G)$ in terms of the same parameters for $G$.  See for example \cite{FiMcBo1998, LWLG2006, ChXi2006, PaZh2010, BaRa2015, AbRa2019, BCKLPR2020}.  In this paper, we will compare the {determining number} and, when relevant, the {cost of 2-distinguishing} for a finite simple graph $G$ to the same parameters for the Mycielskian graphs arising from $G$. These parameters are defined and motivated below.

A coloring of the vertices of a graph $G$ with the colors $1,\ldots, d$ is called a \emph{$d$-distinguishing coloring} if no nontrivial automorphism of $G$ preserves the color classes. A graph is called \emph{$d$-distinguishable} if it has a $d$-distinguishing coloring. The distinguishing number of $G$, denoted $\dist(G)$, is the smallest number of colors necessary for a distinguishing coloring of $G$. Albertson and Collins introduced graph distinguishing in~\cite{AC1996}.  Independently  in 1977 \cite{Ba1977}, Babai
introduced the same definition but called it \emph{asymmetric coloring}.  Here, we will continue the terminology of Albertson and Collins. Most of the work in graph distinguishing in the last few decades has proved that for a large number of graph families, all but a finite number of members are 2-distinguishable. Examples of such families of finite graphs include: hypercubes $Q_n$ with $n\geq 4$~\cite{BC2004}, Cartesian powers $G^n$ for a connected graph $G\ne K_2,K_3$ and $n\geq 2$~\cite{A2005, IK2006,KZ2007}, and Kneser graphs $K_{n:k}$ with $n\geq 6, k\geq 2$~\cite{AB2007}. Examples of such families of infinite graphs include: the denumerable random graph~\cite{IKT2007}, the infinite hypercube~\cite{IKT2007}, and denumerable vertex-transitive graphs of connectivity 1~\cite{STW2012}. 

In 2007, Imrich~\cite{IW} asked whether distinguishing could be refined to provide more information within the class of 2-distinguishable graphs. In response, Boutin~\cite{B2008} defined the \emph{cost of 2-distinguishing} a 2-distinguishable graph $G$ to be the minimum size of a color class over all 2-distinguishing colorings of $G$. The cost of 2-distinguishing $G$ is denoted $\rho(G)$.

Some of the graph families with known or bounded cost are hypercubes with $\lceil \log_2 n \rceil {+} 1 \leq \rho(Q_n) \leq 2\lceil \log_2 n \rceil {-}1$ for $n\geq 5$ \cite{B2008}, Kneser graphs with $\rho(K_{2^{m}-1:2^{m{-}1}{-}1}) = m{+}1$ \cite{B2013b}, and $\rho(K_{2^m} \Box H) = m\cdot 2^{m{-}1}$, where $\Box$ denotes the Cartesian product and $H$ is a graph with no nontrivial automorphisms~\cite{BI2017}. 
 
A determining set is a useful tool in finding the distinguishing number and, when relevant, the cost of 2-distinguishing. A subset $S\subseteq V(G)$ is said to be a \emph{determining set} for $G$ if the only automorphism that fixes the elements of $S$ pointwise is the trivial automorphism. Equivalently, $S$ is a determining set for $G$ if whenever $\varphi$ and $\psi$ are automorphisms of $G$ with $\varphi(x)=\psi(x)$ for all $x\in S$, then $\varphi=\psi$~\cite{B2006}. The \emph{determining number} of a graph $G$, denoted $\det(G)$, is the size of a smallest determining set. Intuitively, if we think of automorphisms of a graph as allowing vertices to move among their relative positions, one can think of the determining number as the smallest number of pins needed to \lq\lq pin down" the graph. 

For some graph families, we only have bounds on the determining number. For instance, for Kneser graphs, $\log_2 (n{+}1)\leq \det(K_{n:k}) \leq n{-}k$ with both upper and lower bounds sharp \cite{B2006}. However, there are families for which we know the determining numbers of its members exactly. In particular, for hypercubes, $\det(Q_n){=}\lceil \log_2 n \rceil {+}1$, and for Cartesian powers $\det(K_3^n){=} \lceil \log_3(2n{+}1) \rceil {+}1$~\cite{B2009}. 

Though distinguishing numbers and determining numbers were introduced by different people and for different purposes, they have strong connections. Albertson and Boutin showed in~\cite{AB2007} that if $G$ has a determining set $S$ of size $d$, then giving each vertex in $S$ a distinct color from $1,\ldots,d$ and every other vertex color $d{+}1$ gives a $(d{+}1)$-distinguishing coloring of $G$. Thus, $\dist(G) \leq \det(G) {+}1$.

Further, in~\cite{B2013}, Boutin pointed out that, given a 2-distinguishing coloring of $G$, since only the trivial automorphism preserves the color classes setwise, only the trivial automorphism preserves them pointwise. Consequently, each of the color classes in a 2-distinguishing coloring is a determining set for the graph, though not necessarily of minimum size. Thus, if $G$ is 2-distinguishable, then $\det(G) \leq \rho(G)$.

In 2018, Alikhani and Soltani~\cite{AS2018} studied the distinguishing number of the traditional Mycielskian of a graph. In particular, they showed that the classic Mycielski graphs $M_k = \mu^k(K_2)$ satisfy $\dist(M_n) = 2$ for all $n \ge 2$. 

To generalize to Mycielskians of arbitrary graphs, they considered the role of twin vertices. Two vertices in a graph are said to be \emph{twins} if they have the same open neighborhood, and a graph is said to be \emph{twin-free} if it does not contain any twins. In particular, Alikhani and Soltani proved that if $G$ is twin-free with at least two vertices, then $\dist(\mu(G)) \le \dist(G){+}1$. Further, they conjectured that for all but a finite number of connected graphs $G$ with at least 3 vertices, $\dist(\mu(G)) \le \dist(G)$. In \cite{BCKLPR2020}, Boutin, Cockburn, Keough, Loeb, Perry, and Rombach proved the conjecture with the theorem stated below. Notice that this theorem does not require graph connectedness. 

 \begin{thm}\label{thm:AS}\cite{BCKLPR2020}
 Let $G\neq K_1, K_2$ be a graph with $\ell \geq 0$ isolated vertices. If $t \ell > \dist(G)$, then $\dist(\mu_t(G)) = t \ell$. Otherwise, 
 $\dist(\mu_t(G))\leq \dist(G)$. \end{thm}

As seen in Theorem~\ref{thm:AS} and \cite{LWLG2006,BR2008}, the presence of isolated vertices in $G$ has a significant effect on the structure and behavior of $\gm$. If $v_i$ is an isolated vertex in $G$ and $t \ge 2$, then in $\mu_t(G)$, $v_i = u_i^0, u_i^1, \dots u_i^{t{-}1}$ are all isolated vertices and, hence, mutually twins. Thus, it is common to exclude isolated vertices when studying the Mycielski constructions. In this paper, we will restrict our attention to finite simple graphs that are not necessarily connected, but that have no isolated vertices. Analogous results for Mycielskians of graphs with isolated vertices, including $K_1$, are covered in a forthcoming paper. 

In Section~\ref{sec:autos}, we give the formal definitions of the traditional and generalized Mycielskians of a graph, and provide lemmas regarding the action of their automorphisms. The twin-free case is considered in Section~\ref{sec:det&rhoTwinFree}. In particular, we prove that if a graph $G$ is not $K_2$ and is twin-free with no isolated vertices, then \[ \det(\gm) = \det(G). \] We also show that if, in addition to the above hypotheses, $\det(G) \geq 2$ and $t\ge \det(G)-1$, then \[\dist(\gm)=2 \text{ \ and \ } \rho(\gm)=\det(G). \]
 
Finally, in Section~\ref{sec:det&rhoTwins}, we prove results on the determining number for the Mycielski construction applied to graphs with twins. To accomplish this we develop a technique utilizing equivalence classes of twins and the resulting quotient graph. This technique allows us to show that the presence of twin vertices in $G$ causes the determining number of $\gm$ to grow linearly with the number of layers, $t$. More precisely, if $G$ is a graph with twins and $T$ is the set consisting of all but one vertex from each set of mutually twin vertices in $G$, then

\[\det(\gm) = t|T| + \det(G) .\]


\section{Generalized Mycielskian Graphs} \label{sec:autos}

In this section, we formally define the traditional and generalized Mycielskians of a graph and we present observations about twin vertices in these graphs. We then discuss how automorphisms of the Mycielskian graph behave when the underlying graph is twin-free and has no isolated vertices.

Throughout this paper, let $N_G(v)$ be the open neighborhood of $v$ in $G$. For ease of notation, the open neighborhood of $v$ in the Mycielskian graph will simply be denoted $N(v)$.

Suppose $G$ is a finite simple graph with $V(G) = \{v_1,\ldots,v_n\}$. The \emph{Mycielskian of} $G$, denoted $\mu(G)$, is a graph with vertex set 
\[
V(\mu(G)) = \{ v_1, \dots, v_n, u_1, \dots u_n, w\}.
\]
 For each edge $v_iv_j \in E(G)$, $\mu(G)$ has edges $v_iv_j$, $u_iv_j$ and $v_iu_j$; and additionally, $u_iw$ for all $1 \le i \le n$. That is, $N(u_i)=N_G(v_i)\cup\{w\}$ and $\mu(G)$ contains $G$ as an induced subgraph on the vertex set $\{ v_1,\dots,v_n\}$. We refer to the vertices $v_1, \dots , v_n$ in $\mu(G)$ as {\em original vertices} and the vertices $u_1, \dots , u_n$ as {\em shadow vertices}. 
 We refer to vertex $w$ as the {\em root}. Notice that the shadow vertices form an independent set. See Figure~\ref{fig:GenMyExamplesTrad} for drawings of $\mu(K_2)$ and $\mu(K_3)$.

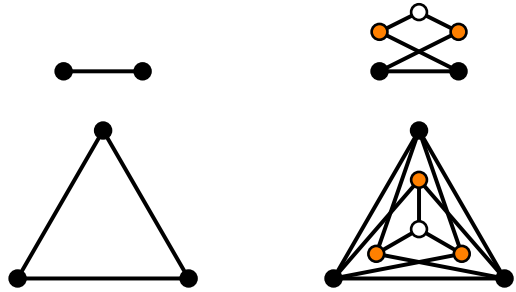
\begin{figure}[h]
\centering
 \begin{tikzpicture}[scale=.7]
 \begin{scope}[shift={(0,0)},scale=.75]
  \draw[black!100,line width=1.5pt] (-1,0) -- (1,0); 
 \draw[fill=black!100,line width=1] (-1,0) circle (.2);
 \draw[fill=black!100,line width=1] (1,0) circle (.2);
    \end{scope}
    
\begin{scope}[shift={(6,0)},rotate=180,scale=.75]
\draw[black!100,line width=1.5pt] (-1,0) -- (1,0); 
\draw[black!100,line width=1.5pt] (-1,0) -- (1,-1);
\draw[black!100,line width=1.5pt] (-1,-1) -- (1,0);
\draw[black!100,line width=1.5pt] (-1,-1) -- (0,-1.5);
\draw[black!100,line width=1.5pt] (1,-1) -- (0,-1.5);
\draw[fill=black!100,line width=1] (-1,0) circle (.2);
\draw[fill=black!100,line width=1] (1,0) circle (.2);
\draw[fill=orange!100,line width=1] (-1,-1) circle (.2);
\draw[fill=orange!100,line width=1] (1,-1) circle (.2);
\draw[fill=white!100,line width=1] (0,-1.5) circle (.2);
\end{scope}

\begin{scope}[shift={(0,-3)},scale=.75]
\draw[black!100,line width=1.5pt] (-30:2.5) -- (-150:2.5);
\draw[black!100,line width=1.5pt] (90:2.5) -- (-150:2.5);
\draw[black!100,line width=1.5pt] (-30:2.5) -- (90:2.5);
 \draw[fill=black!100,line width=1] (-30:2.5) circle (.2);
 \draw[fill=black!100,line width=1] (90:2.5) circle (.2);
 \draw[fill=black!100,line width=1] (-150:2.5) circle (.2);
\end{scope}

\begin{scope}[shift={(6,-3)},scale=.75]
\draw[black!100,line width=1.5pt] (-30:2.5) -- (-150:2.5);
\draw[black!100,line width=1.5pt] (90:2.5) -- (-150:2.5);
\draw[black!100,line width=1.5pt] (-30:2.5) -- (90:2.5);
\draw[black!100,line width=1.5pt] (-30:2.5) -- (-150:1.25);
\draw[black!100,line width=1.5pt] (90:2.5) -- (-150:1.25);
\draw[black!100,line width=1.5pt] (-30:2.5) -- (90:1.25);
\draw[black!100,line width=1.5pt] (-30:1.25) -- (-150:2.5);
\draw[black!100,line width=1.5pt] (90:1.25) -- (-150:2.5);
\draw[black!100,line width=1.5pt] (-30:1.25) -- (90:2.5);
\draw[black!100,line width=1.5pt] (-30:1.25) -- (0,0);
\draw[black!100,line width=1.5pt] (90:1.25) -- (0,0);
\draw[black!100,line width=1.5pt] (-150:1.25) -- (0,0);
 \draw[fill=black!100,line width=1] (-30:2.5) circle (.2);
 \draw[fill=black!100,line width=1] (90:2.5) circle (.2);
 \draw[fill=black!100,line width=1] (-150:2.5) circle (.2);
 \draw[fill=orange!100,line width=1] (-30:1.25) circle (.2);
 \draw[fill=orange!100,line width=1] (90:1.25) circle (.2);
 \draw[fill=orange!100,line width=1] (-150:1.25) circle (.2);
 \draw[fill=white!100,line width=1] (0,0) circle (.2);
\end{scope}
\end{tikzpicture}

\caption{Top: $K_2$ and $\mu(K_2)$ drawn with vertical levels with the root at the top. Bottom: $K_3$ and $\mu(K_3)$ drawn with concentric levels with the root in the middle.}

\label{fig:GenMyExamplesTrad}

\end{figure}

We now establish an observation about the relationship between automorphisms of graphs and automorphisms of their Mycielskians.

\begin{obs}\label{obs:auts} \it An automorphism $\alpha$ of $G$ induces an automorphism $\widehat \alpha$ on $\mu(G)$ by replicating the action of $\alpha$ on the set of original vertices and also on the set of shadow vertices, and leaving the root fixed. That is, we can define the automorphism $\widehat \alpha$ of $\mu(G)$ by $\widehat\alpha(v_i) = \alpha(v_i) = v_j$ and $\widehat \alpha(u_i) = u_j$ and $\widehat\alpha(w)=w$. \end{obs}

We will see that for most graphs, the only automorphisms of $\mu(G)$ are those induced by automorphisms of $G$. Making this more precise involves a discussion of twin vertices. Recall from the introduction that vertices $x$ and $y$ are called twins if they have precisely the same set of neighbors. It is possible to have a collection of three or more mutually twin vertices. If two vertices of $G$ are twins, then it is straightforward to show that exchanging the twins and fixing the remaining vertices is an automorphism of $G$. Recall from the introduction that a vertex set $S$ is a determining set for $G$ if the only automorphism that fixes the elements of $S$ pointwise is the trivial automorphism. Thus, every determining set must contain all but one representative from every collection of mutually twin vertices. 

The following are some structural observations about relationships between twins in $G$ and twins in $\mu(G)$.

\begin{obs}\label{obs:twins1}\it
In $\mu(G)$, the root $w$ is adjacent to each shadow vertex, but not adjacent to any original vertex, so twins in $\mu(G)$ are either both shadow vertices or both original vertices.\end{obs} 

\begin{obs} \label{obs:twins2}\it If 
$v_i$ and $v_j$ are twins in $G$, then they are twins in $\mu(G)$ and so are their shadows $u_i$ and $u_j$. \end{obs} 

\begin{obs} \label{obs:twins3} \it If at least one pair of $\{v_i,v_j\}$ or $\{u_i,u_j\}$ is twins in $\mu(G)$, then both pairs are, as is $\{v_i,v_j\}$ within $G$. 
\end{obs}

In~\cite{AS2018}, Alikhani and Soltani considered automorphisms of $\mu(G)$ when $G$ is twin-free, and proved the following.

\begin{lemma} \label{lem:A&Sfixed}
\cite{AS2018} If $G$ is twin-free and $\widehat\alpha$ is an automorphism of $\mu(G)$ that fixes the root, then 
\begin{enumerate}[(i)] 
\item\label{lem:A&Sfixed:levels} $\widehat\alpha$ preserves the set of original vertices, $\{v_1,\ldots,v_n\}$, and the set of shadow vertices, $\{u_1,\ldots,u_n\}$;
\item\label{lem:A&Sfixed:restrict} $\widehat\alpha$ restricted to $\{v_1, \dots, v_n\}$ is an automorphism $\alpha$ of $G$;
\item\label{lem:A&Sfixed:shadows} $\alpha(v_i) = v_j$ if and only if $\widehat\alpha(u_i) = u_j$.
\end{enumerate}
\end{lemma}

Thus, if $G$ is twin-free, then every automorphism of $\mu(G)$ that fixes the root is induced by an automorphism of $G$.

The \emph{generalized Mycielskian} of $G$, also known as a \emph{cone over} $G$, was defined by Stiebitz~\cite{stiebitz1985beitrage} in 1985 (cited in~\cite{ct2001}) and independently by Van Ngoc~\cite{van1987graph} in 1987 (cited in~\cite{van19954}), and has multiple levels of shadow vertices. More precisely, for $t \geq 1$, the generalized Mycielskian of $G$, denoted $\mu_t(G)$, has vertex set
 
 \[\{u^0_1,\ldots,u^0_n,u^1_1,\ldots,u^1_n,\dots,u^t_1,\ldots,u^t_n,w\}.\]
 
For each edge $v_i v_j$ in $G$, the graph $\mu_t(G)$ has edge $v_i v_j = u_i^0u_j^0$, as well as edges $u^s_iu^{s{+}1}_j$ and $u^s_j u^{s+1}_i$, for $0\leq s <t$. Finally, $\mu_t(G)$ has edges $u^t_i w$ for $1 \leq i \leq n$. Intuitively, for $1\leq s \leq t{-}1$, the neighbors of the shadow vertex $u_i^s$ are the shadows of the neighbors of $v_i$ both at level $s{-}1$ and at level $s{+}1$, while the neighbors of $u_i^t$ are the shadows of the neighbors of $v_i$ at level $t{-}1$ and the root $w$.

We say that vertex $u_i^s$ is \emph{at level $s$} and we call level $t$ the \emph{top level}. In addition, we make the identification $u_i^0 = v_i$. As we did for the traditional Mycielskian, we refer to the vertices at level $0$ as \emph{original vertices}, to the vertices at level $1\le s\le t$ as \emph{shadow vertices}, and to $w$ as the \emph{root}. Since $\mu_1(G) = \mug$, we omit the subscript when $t=1$. See Figure~\ref{fig:GenMyExamplesGen} for drawings of $\mu_2(K_2)$ and $\mu_2(K_3)$.

\begin{figure}[h]
\centering
 \begin{tikzpicture}[scale=.7]
 \begin{scope}[shift={(0,0)},scale=.75]
  \draw[black!100,line width=1.5pt] (-1,0) -- (1,0); 
 \draw[fill=black!100,line width=1] (-1,0) circle (.2);
 \draw[fill=black!100,line width=1] (1,0) circle (.2);
    \end{scope}
    
    \begin{scope}[shift={(6,0)},rotate=180, scale =.75]
\draw[black!100,line width=1.5pt] (-1,0) -- (1,0); 
\draw[black!100,line width=1.5pt] (-1,0) -- (1,-1);
\draw[black!100,line width=1.5pt] (-1,-1) -- (1,0);
\draw[black!100,line width=1.5pt] (-1,-1) -- (1,-2);
\draw[black!100,line width=1.5pt] (-1,-2) -- (1,-1);
\draw[fill=black!100,line width=1] (-1,0) circle (.2);
\draw[black!100,line width=1.5pt] (-1,-2) -- (0,-2.5);
\draw[black!100,line width=1.5pt] (1,-2) -- (0,-2.5);
\draw[fill=black!100,line width=1] (1,0) circle (.2);
\draw[fill=orange!100,line width=1] (-1,-1) circle (.2);
\draw[fill=orange!100,line width=1] (1,-1) circle (.2);
\draw[fill=yellow!100,line width=1] (-1,-2) circle (.2);
\draw[fill=yellow!100,line width=1] (1,-2) circle (.2);
\draw[fill=white!100,line width=1] (0,-2.5) circle (.2);
\end{scope}

\begin{scope}[shift={(0,-3)},scale=.75]
\draw[black!100,line width=1.5pt] (-30:2.5) -- (-150:2.5);
\draw[black!100,line width=1.5pt] (90:2.5) -- (-150:2.5);
\draw[black!100,line width=1.5pt] (-30:2.5) -- (90:2.5);
 \draw[fill=black!100,line width=1] (-30:2.5) circle (.2);
 \draw[fill=black!100,line width=1] (90:2.5) circle (.2);
 \draw[fill=black!100,line width=1] (-150:2.5) circle (.2);
\end{scope}

\begin{scope}[shift={(6,-3)},scale=.75]
\draw[black!100,line width=1.5pt] (-30:2.5) -- (-150:2.5);
\draw[black!100,line width=1.5pt] (90:2.5) -- (-150:2.5);
\draw[black!100,line width=1.5pt] (-30:2.5) -- (90:2.5);
\draw[black!100,line width=1pt] (-30:2.5) -- (-150:1.66);
\draw[black!100,line width=1pt] (90:2.5) -- (-150:1.66);
\draw[black!100,line width=1pt] (-30:2.5) -- (90:1.66);
\draw[black!100,line width=1pt] (-30:1.66) -- (-150:2.5);
\draw[black!100,line width=1pt] (90:1.66) -- (-150:2.5);
\draw[black!100,line width=1pt] (-30:1.66) -- (90:2.5);
\draw[black!100,line width=1pt] (-30:.83) -- (-150:1.66);
\draw[black!100,line width=1pt] (90:.83) -- (-150:1.66);
\draw[black!100,line width=1pt] (-30:.83) -- (90:1.66);
\draw[black!100,line width=1pt] (-30:1.66) -- (-150:.83);
\draw[black!100,line width=1pt] (90:1.66) -- (-150:.83);
\draw[black!100,line width=1pt] (-30:1.66) -- (90:.83);
\draw[black!100,line width=1pt] (-30:.83) -- (0,0);
\draw[black!100,line width=1pt] (90:.83) -- (0,0);
\draw[black!100,line width=1pt] (-150:.83) -- (0,0);
 \draw[fill=black!100,line width=1] (-30:2.5) circle (.2);
 \draw[fill=black!100,line width=1] (90:2.5) circle (.2);
 \draw[fill=black!100,line width=1] (-150:2.5) circle (.2);
 \draw[fill=yellow!100,line width=1] (-30:.83) circle (.2);
 \draw[fill=yellow!100,line width=1] (90:.83) circle (.2);
 \draw[fill=yellow!100,line width=1] (-150:.83) circle (.2);
  \draw[fill=orange!100,line width=1] (-30:1.66) circle (.2);
 \draw[fill=orange!100,line width=1] (90:1.66) circle (.2);
 \draw[fill=orange!100,line width=1] (-150:1.66) circle (.2);
 \draw[fill=white!100,line width=1] (0,0) circle (.2);
\end{scope}

\end{tikzpicture}

\caption{Top: $K_2$ and $\mu_2(K_2)$ drawn with vertical levels with the root at the top. Bottom: $K_3$ and $\mu_2(K_3)$ drawn with concentric levels with the root in the middle.}
\label{fig:GenMyExamplesGen}

\end{figure}
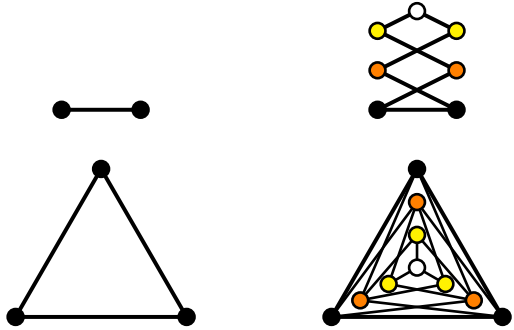

As one might suspect, Observations \ref{obs:auts}, \ref{obs:twins1}, \ref{obs:twins2}, and \ref{obs:twins3}, as well as Lemma \ref{lem:A&Sfixed}, for traditional Mycielskian graphs extend to the generalized Mycielskians with only minor changes. 

\begin{obs}\label{obs:auts'} \it
 An automorphism $\alpha$ of $G$ induces an automorphism $\widehat \alpha$ on $\gm$ by replicating the action of $\alpha$ on each level of $\gm$. That is, we can define the automorphism $\widehat \alpha$ of $\gm$ by $\widehat\alpha(u^0_i) = \alpha(v_i) = v_j$ and $\widehat \alpha(u^s_i) = u^s_j$ for each $1\leq s \leq t$ and $\widehat\alpha(w)=w$. 
\end{obs}

\begin{obs}\label{obs:twins1'} \it
If $G$ has no isolated vertices, twin vertices in $\mu_t(G)$ must be vertices at the same level. That is, twin vertices have the form $u_i^s$ and $u_j^s$ for some $0 \le s \le t$. 
\end{obs}

\begin{obs}\label{obs:twins2'} \it
If original vertices $v_i$ and $v_j$ are twins in $G$, then $u_i^s$ and $u_j^s$ are twins in $\mu_t(G)$ for all $0 \le s \le t$.
\end{obs}

\begin{obs}\label{obs:twins3'}\it
If $\{u_i^s,u_j^s\}$ are twins in $\mu_t(G)$ for any $0 \le s \le t$, then they are twins for all such $s$. In particular, $\{v_i, v_j\}$ are twins in $G$.
\end{obs}

Boutin, Cockburn, Keough, Loeb, Perry and Rombach prove in \cite{BCKLPR2020} that if $G$ is not $K_{1,m}$ for $m\geq 0$, then all automorphisms of $\mu_t(G)$ fix the root.

\begin{lemma} \label{lem:phiwisv} \cite{BCKLPR2020}
Let $G$ be a graph and let $t \geq 1$. Let $\widehat \alpha$ be an automorphism of $\mu_t(G)$.
\begin{enumerate}[(i)] 
  \item \label{lem:phiwisv:i}If $G = K_{1,1} = K_2$, then $\mu_t(G) = C_{2t{+}3}$, and $\widehat\alpha(w)$ can be any vertex. 
 \item \label{lem:phiwisv:ii}If $G=K_{1,m}$ with $m \neq 1$ and $v$ is the vertex of degree $m$ in $G$, then $\widehat\alpha(w) \in \{w,u^t\}$, where $u^t$ is the top level shadow of $v$.
  \item \label{lem:phiwisv:iii} If $G \neq K_{1,m}$ for any $m \ge 0$, then $\widehat\alpha(w) = w$.
\end{enumerate}
\end{lemma}

In essence, Lemma~\ref{lem:phiwisv} states that the only graphs $G$ that have automorphisms of $\gm$ that do not fix $w$ are the star graphs $K_{1,m}$. Figure~\ref{fig:mutK13} shows the graphs $ K_{1,3}$ and $\mu_t(K_{1,3})$ for $t = 1, 2, 3$. The vertical reflectional symmetries in the drawings of $\mu_t(K_{1,3})$ in Figure~\ref{fig:mutK13} correspond to automorphisms that move the root to the top level shadow of the central vertex in $K_{1,3}$.

\begin{figure}[h]
  \centering
   \begin{tikzpicture}[scale=.5]
\draw[fill=black!100,line width=1] (0,0) circle (.2);
\foreach \j in {1,...,3}
{\draw[fill=black!100,line width=1] (120*\j:1) circle (.2);\draw[black!100,line width=1.5pt] (120*\j:1) -- (0,0);
}
\begin{scope}[shift={(5,-2)}]
\draw[black!100,line width=1.5pt] (-3,0) -- (3,0); 
\draw[black!100,line width=1.5pt] (-3,4) -- (3,4); 
\draw[black!100,line width=1.5pt] (-3,0) -- (-3,4); 
\draw[black!100,line width=1.5pt] (3,0) -- (3,4); 
\draw[black!100,line width=1.5pt] (0,0) -- (2,1.3); 
\draw[black!100,line width=1.5pt] (0,0) -- (1,2.6); 
\draw[black!100,line width=1.5pt] (3,4) -- (2,1.3); 
\draw[black!100,line width=1.5pt] (3,4) -- (1,2.6);
\draw[black!100,line width=1.5pt] (0,0) -- (-2,1.3); 
\draw[black!100,line width=1.5pt] (0,0) -- (-1,2.6); 
\draw[black!100,line width=1.5pt] (-3,4) -- (-2,1.3);
\draw[black!100,line width=1.5pt] (-3,4) -- (-1,2.6);
\draw[fill=black!100,line width=1] (0,0) circle (.2);
\draw[fill=black!100,line width=1] (-3,0) circle (.2);
\draw[fill=orange!100,line width=1] (3,0) circle (.2);
\draw[fill=white!100,line width=1] (3,4) circle (.2);
\draw[fill=orange!100,line width=1] (-3,4) circle (.2);
\draw[fill=orange!100,line width=1] (2,1.3) circle (.2);
\draw[fill=orange!100,line width=1] (1,2.6) circle (.2);
\draw[fill=black!100,line width=1] (-2,1.3) circle (.2);
\draw[fill=black!100,line width=1] (-1,2.6) circle (.2);
\end{scope}
\begin{scope}[shift={(20,-5)}]
\draw[black!100,line width=1.5pt] (-3,2) -- (3,2); 
\draw[black!100,line width=1.5pt] (-3,8) -- (3,8); 
\draw[black!100,line width=1.5pt] (-3,2) -- (-3,8); 
\draw[black!100,line width=1.5pt] (3,2) -- (3,8); 
\draw[black!100,line width=1.5pt] (0,2) -- (2,2.7); 
\draw[black!100,line width=1.5pt] (0,2) -- (1,3.4); 
\draw[black!100,line width=1.5pt] (3,4) -- (2,2.7); 
\draw[black!100,line width=1.5pt] (3,4) -- (1,3.4);
\draw[black!100,line width=1.5pt] (0,2) -- (-2,2.7); 
\draw[black!100,line width=1.5pt] (0,2) -- (-1,3.4); 
\draw[black!100,line width=1.5pt] (-3,4) -- (-2,2.7);
\draw[black!100,line width=1.5pt] (-3,4) -- (-1,3.4);
\draw[black!100,line width=1.5pt] (-3,4) -- (-2,6) -- (-3,8) -- (-1,6) -- (-3,4);
\draw[black!100,line width=1.5pt] (3,4) -- (2,6) -- (3,8) -- (1,6) -- (3,4);
\draw[fill=black!100,line width=1] (0,2) circle (.2);
\draw[fill=black!100,line width=1] (-3,2) circle (.2);
\draw[fill=purple!100,line width=1] (3,2) circle (.2);
\draw[fill=orange!100,line width=1] (3,4) circle (.2);
\draw[fill=purple!100,line width=1] (-3,4) circle (.2);
\draw[fill=purple!100,line width=1] (2,2.7) circle (.2);
\draw[fill=purple!100,line width=1] (1,3.4) circle (.2);
\draw[fill=black!100,line width=1] (-2,2.7) circle (.2);
\draw[fill=black!100,line width=1] (-1,3.4) circle (.2);
\draw[fill=yellow!100,line width=1] (-3,8) circle (.2);
\draw[fill=white!100,line width=1] (3,8) circle (.2);
\draw[fill=orange!100,line width=1] (-3,6) circle (.2);
\draw[fill=yellow!100,line width=1] (3,6) circle (.2);
\draw[fill=yellow!100,line width=1] (2,6) circle (.2);
\draw[fill=yellow!100,line width=1] (1,6) circle (.2);
\draw[fill=orange!100,line width=1] (-2,6) circle (.2);
\draw[fill=orange!100,line width=1] (-1,6) circle (.2);
\end{scope}
\begin{scope}[shift={(12.5,-6)}]
\draw[black!100,line width=1.5pt] (-3,8) -- (3,8); 
\draw[black!100,line width=1.5pt] (-3,8) -- (-3,4) -- (3,4);
\draw[black!100,line width=1.5pt] (3,4) -- (0,4.7) -- (-3,4); 
\draw[black!100,line width=1.5pt] (-3,4) -- (0,5.4) -- (3,4) -- (3,8); 
\draw[black!100,line width=1.5pt] (-3,4) -- (-2,6) -- (-3,8) -- (-1,6) -- (-3,4);
\draw[black!100,line width=1.5pt] (3,4) -- (2,6) -- (3,8) -- (1,6) -- (3,4);
\draw[fill=black!100,line width=1] (3,4) circle (.2);
\draw[fill=black!100,line width=1] (0,4) circle (.2);
\draw[fill=black!100,line width=1] (0,4.7) circle (.2);
\draw[fill=black!100,line width=1] (0,5.4) circle (.2);
\draw[fill=orange!100,line width=1] (-3,4) circle (.2);
\draw[fill=white!100,line width=1] (-3,8) circle (.2);
\draw[fill=yellow!100,line width=1] (3,8) circle (.2);
\draw[fill=yellow!100,line width=1] (-3,6) circle (.2);
\draw[fill=orange!100,line width=1] (3,6) circle (.2);
\draw[fill=orange!100,line width=1] (2,6) circle (.2);
\draw[fill=orange!100,line width=1] (1,6) circle (.2);
\draw[fill=yellow!100,line width=1] (-2,6) circle (.2);
\draw[fill=yellow!100,line width=1] (-1,6) circle (.2);
\end{scope}
\end{tikzpicture}
  \caption{The graphs $ K_{1,3}$ and $\mu_t(K_{1,3})$ for $t = 1, 2, 3$. In each, the vertices of matching color are at the same level.} \label{fig:mutK13}
\end{figure}
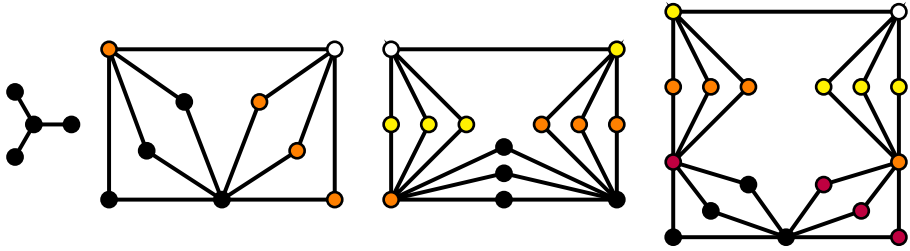

The next lemma is an extension of Lemma~\ref{lem:A&Sfixed} to generalized Mycielskian graphs.

 \begin{lemma} \label{lem:GenA&Sfixed}
 Let $G \neq K_{1,m}$ for any $m \geq 0$ be a graph on $n$ vertices that has no isolated vertices and let $t \geq 1$. If $\widehat \alpha$ is an automorphism of $\mu_t(G)$, then
 \begin{enumerate}[(i)]
 \item\label{lem:GenA&Sfixed:levels} $\widehat \alpha$ preserves the level of vertices, so, $\widehat\alpha(\{u_1^s \dots , u_n^s\}) = \{u_1^s \dots , u_n^s\}$ for all $0 \le s \le t$;
 \item\label{lem:GenA&Sfixed:restrict} $\widehat\alpha$ restricted to $\{u_1^0, \dots , u_n^0\} = \{v_1, \dots, v_n\}$ is an automorphism $\alpha$ of $G$;
 \item\label{lem:GenA&Sfixed:shadows} if, in addition, $G$ is twin-free, then $\alpha(v_i) = v_j$ if and only if $\widehat \alpha(u_i^s) = u_j^s$ for all $0 <s \le t$.
 \end{enumerate}
 \end{lemma}

 \begin{proof}
Since $G \neq K_{1,m}$ for any $m\geq 0$, by Lemma~\ref{lem:phiwisv}(\ref{lem:phiwisv:iii}), any automorphism of $\gm$ fixes the root. Furthermore, since $G$ has no isolated vertices, the level of every vertex $x$ in $\gm$ can be  determined by its distance from $w$. That is, in $\mu_t(G)$, $x = w$ if and only if the distance of $x$ from $w$ is $0$, and $x$ is at level $s$ if and only if its distance from $w$ is {$t {+} 1 {-} s$}. 
  Since distance is preserved by automorphisms, any automorphism that fixes the root preserves the level of every other vertex. This proves~(\ref{lem:GenA&Sfixed:levels}) and~(\ref{lem:GenA&Sfixed:restrict}).

 For (\ref{lem:GenA&Sfixed:shadows}),  we mirror and extend the argument given by Alikhani and Soltani in~\cite{AS2018}. Focusing first on $\alpha$, if $\alpha(v_i)= v_j$, then $\alpha$ maps the open neighborhood of $v_i$ to the open neighborhood of $v_j$. That is, 
$\alpha(v_i)= v_j$ implies $
 \alpha(N_G(v_i))= N_G(v_j).$
 Further, since $G$ is twin-free, every vertex in $G$ can be uniquely identified by its open neighborhood. Thus, $\alpha(N_G(v_i)) = N_G(v_j)$ implies $\alpha(v_i)= v_j$. This gives us the biconditional statement $\alpha(N_G(v_i)) = N_G(v_j)$ if and only if $\alpha(v_i)= v_j$.

 Since $G$ has no isolated vertices and is twin-free, Observations~\ref{obs:twins1'} and \ref{obs:twins3'} imply that $\gm$ is also twin-free, so each vertex of $\gm$ can be uniquely identified by its open neighborhood. Hence, $\widehat \alpha(u_i^s) = u_j^s$ if and only if 
 $\widehat\alpha(N(u_i^s)) = N(u_j^s)$. By definition, 
 \[
 N(u_i^s) = \begin{cases}
 \{ u_k^0, u_k^1 \mid v_k \in N_G(v_i)\}, \quad &\text{ if } s=0,\\
 \{ u_k^{s{-}1}, u_k^{s{+}1} \mid v_k \in N_G(v_i)\}, \quad &\text{ if } 0 < s < t,\\
  \{ u_k^{t{-}1} \mid v_k \in N_G(v_i)\} \cup \{w\}, &\text{ if } s=t.
 \end{cases}
 \]
 Thus, the open neighborhood of $u_i^s$ in $\gm$ is completely determined by the open neighborhood of $u_i^0 = v_i$ in $G$. This, in turn, implies that $\widehat \alpha(N(u_i^s)) = N(u_j^s)$ if and only if $\alpha(N_G(v_i)) = N_G(v_j)$. Together with the biconditional statement from the previous paragraph~(\ref{lem:GenA&Sfixed:shadows}) follows.\end{proof}


\section{\boldmath ${\det}(\mu_t(G))$ and $\rho(\mu_t(G))$ for Twin-Free $G$} \label{sec:det&rhoTwinFree}

We begin this section by proving that for most twin-free graphs, the generalized Mycielski construction preserves determining number, in contrast to the proven effect of the Mycielski construction on distinguishing number. More explicitly, in \cite{BCKLPR2020}, the current authors proved that for $G\neq K_1, K_2$, $\dist(\mu(G))\leq \dist(G)$, and that these values may be arbitrarily far apart. For example, if $n \ge 3$ and $t \ge \log_2 (n-1)$, then $\dist(\mu_t(K_n)) = 2$, whereas $\dist(K_n) = n$.

\begin{thm}\label{thm:twinfreeDet}

 Let $G$ be a twin-free graph with no isolated vertices and let $t\geq 1$.
\begin{enumerate}[(i)]
  
   \item\label{thm:twinfreeDet:K2} If $G = K_2$, then $\det(G) = 1$ and $\det(\gm) = 2.$ 
   
  \item\label{thm:twinfreeDet:notK2} If $G \neq K_2$, then any (minimum size) determining set for $G$ is a (minimum size) determining set for $\gm$ and \[\det(\gm) = \det(G).\]
\end{enumerate}

\end{thm}

\begin{proof}

For~(\ref{thm:twinfreeDet:K2}), if $G=K_2$, then $\det(G) = 1$ and $\det(\gm) = 2$, since $\mu_t(K_2) = C_{2t+3}$.

For~(\ref{thm:twinfreeDet:notK2}), let $S \subseteq V(G)$ be a determining set for $G$; we will show $S$, as a subset of $V(\gm)$, is also a determining set for $\mu_t(G)$. Let $\widehat \alpha\in\aut(\gm)$ and assume that $\widehat\alpha(s) = s$ for all $s \in S$. By Lemma~\ref{lem:GenA&Sfixed}(\ref{lem:GenA&Sfixed:restrict}), the restriction of $\widehat \alpha$ to $V(G)$ is an automorphism $\alpha\in\aut(G)$ and by Lemma~\ref{lem:GenA&Sfixed}(\ref{lem:GenA&Sfixed:shadows}), $\alpha(v_i) = v_j$ if and only if $\widehat\alpha(u_i^s) = u_j^s$ for all $0 \le s \le t$. By the assumptions that $\widehat \alpha$ fixes $S$ pointwise and that $S$ is a determining set for $G$, $\alpha$ is the identity on the original vertices. Thus, $\widehat \alpha(v_i)=\alpha(v_i) = v_i$ for all $v_i\in V(G)$ and hence, $\widehat\alpha(u_i^s) = u_i^s$ for all $i$ and all $s$. Thus, $S$ is a determining set for $\gm$ and so $\det(\gm)\leq \det(G)$.

Now, suppose instead that $S \subseteq V(\gm)$ is a minimum size determining set for $\gm$. Since $G$ is twin-free, has no isolated vertices, and is not $K_2$, by Lemma~\ref{lem:phiwisv}(\ref{lem:phiwisv:iii}), every automorphism of $\gm$ fixes the root $w$ and so, by the minimality of $S$, we can assume $w \notin S$. Let \[S_0 = \{v_i \in V(G) \mid u_i^s \in S \text{ for some } 0 \le s \le t \}.\] Then $|S_0| \le |S|$. If $\beta\in\aut(G)$ fixes $S_0$ pointwise, then by Lemma~\ref{lem:GenA&Sfixed}(\ref{lem:GenA&Sfixed:shadows}), the automorphism $\widehat\beta$ on $\gm$ fixes $S$. 
Thus, $\widehat\beta$ is the identity on $\gm$ and so restricts to the identity on $G$. Hence, $S_0$ is a determining set for $G$ and so $\det(G)\leq \det(\gm)$, yielding equality.\end{proof}

Now we consider graphs with small determining number. By definition, $\det(G) = 0$ if and only if $G$ has only the trivial automorphism, or, equivalently, if $G$ is asymmetric. In particular, $\det(G) = 0$ if and only if $\dist(G) = 1$. If $G$ has nontrivial automorphisms, then $\det(G) = 1$ if and only if $G$ has a vertex $x$ that forms a singleton determining set. In this case, we can color $x$ red and all other vertices blue to obtain a 2-distinguishing coloring of $G$. This coloring shows that if $\det(G)=1$, then $\dist(G) = 2$ and $\rho(G)=1$. Note that these facts hold for graphs with or without twins.

If $G\neq K_2$ is a twin-free graph with no isolated vertices and $\det(G) = 1$, then by Theorem~\ref{thm:twinfreeDet}, $\det(\mu_t(G))=\det(G) = 1$. Thus, since $\det(\gm)=1$, by the reasoning in the previous paragraph, $\dist(\gm)=2$ and $\rho(\gm)=1$.

 More generally, if $G$ is a twin-free graph with no isolated vertices, then  $\gm$ is 2-distinguishable provided $t$ is sufficiently large relative to the determining number of $G$. The following two theorems make this more precise.

\begin{thm}\label{thm:tfdetdist} 
Let $G$ be a twin-free graph with no isolated vertices such that $\det(G) \geq 2$. Then for $t \geq \lceil \log_2(\det(G)+1)\rceil -1$, 
\[ 
\det(\gm) = \det(G), \dist(\gm) = 2, \text{ and } \]

\[\rho(\gm) \leq \frac{(\det(G)+1) \lceil \log_2(\det(G)+1) \rceil}{2}.
\]\end{thm}

\begin{proof} Let $\det(G)=k$. Since $\det(K_2) = 1$ and we are assuming $\det(G) \geq 2$ we have $G\ne K_2$. Thus, by Theorem~\ref{thm:twinfreeDet}, $\det(\mu_t(G)) =\det(G) = k$. Further, since $G$ is twin-free and has no isolated vertices, $G \neq K_{1,m}$ for any $m\ge 0$.

To show $\dist(\gm) = 2$, first observe that since $\det(\gm) \geq 2$, $\gm$ has a nontrivial automorphism and so cannot be distinguished with one color. Thus, $\dist(\gm) \geq 2$.

We will now show $\dist(\gm)\leq 2$. Let $r=\lceil \log_2 (k+1)\rceil$ and let $S = \{v_1, v_2, \dots, v_k\}$ be a determining set for $G$. For $1 \leq i \leq k$, let $b_1b_2\dots b_r$ be the binary representation of $i$ with leading zeros if necessary. For each $1\leq i \leq k$ and each $0\le j\le r-1$, color $u_i^j$ red if $b_j=1$ and $u_i^j$ blue if $b_j=0$. We color all other vertices blue. Assume $\widehat\alpha \in \aut(\mu_t(G))$ preserves the red and blue color classes. By Lemma~\ref{lem:GenA&Sfixed}(\ref{lem:GenA&Sfixed:levels}), since $G$ is not a star graph, $\widehat \alpha$ preserves levels. Furthermore, by Lemma~\ref{lem:GenA&Sfixed}(\ref{lem:GenA&Sfixed:shadows}) we know that $\widehat\alpha(v_i) = v_j$ if and only if $\widehat \alpha(u_i^s) = u_j^s$ for all $0 <s \le t$. For every $u_i^0=v_i \in S$, the distinct sequence of colors in the ordered set $\{ u^0_i,\dots,u_i^r \}$ guarantees that $\widehat\alpha(v_i) = v_i$ for all $1 \le i \le k$. By the fact that $S$ is a determining set for $G$, and by Lemma~\ref{lem:GenA&Sfixed}(\ref{lem:GenA&Sfixed:restrict}) and~(\ref{lem:GenA&Sfixed:shadows}), we now have that $\widehat\alpha$ fixes the vertices at level $0$ and therefore fixes every vertex in $\gm$. This shows that $\dist(\gm) \leq 2$ and completes the proof that $\dist(\gm) = 2$. Furthermore, this coloring has no more than $(k+1)\frac{r}{2} = (k+1) \lceil \log_2(k+1) \rceil/2$ red vertices, which gives us the upper bound on $\rho (\gm)$.\end{proof}

Note that the bound on $t$ in Theorem~\ref{thm:tfdetdist} is sharp. To see this, consider $G=K_5$ and $t=1$. Since $\det (K_5)=4$, we have $t<\lceil\log_2(4+1)\rceil-1$. By Lemma~\ref{lem:A&Sfixed}, every automorphism of $\mu(K_5)$ fixes $w$ and acts on original-shadow vertex pairs of the form $(v_i, u_i).$ Since $K_5$ is vertex-transitive, there is an automorphism of $\mu(K_5)$ taking each pair $(v_i,u_i)$ to any other pair $(v_j,u_j)$. Thus, in a 2-distinguishing coloring of $\mu(K_5)$, each of the five vertex pairs must be assigned a distinct ordered pair of colors. 

However, since there are precisely four ways in which we can 2-color a pair $(v_i,u_i)$, we see that in any 2-coloring of $\mu(K_5)$ at least two of the vertex pairs must have the same coloring. Thus, no 2-coloring of $\mu(K_5)$ can be distinguishing. So we see that the bound $t\geq \lceil\log_2(4+1)\rceil-1$ is sharp.

The upper bound on $\rho (\gm)$ is sharp as well. To see this, consider $G=K_4$ and $t=1$. Since $\det (K_4)=3$ and $t \ge \lceil\log_2(3+1)\rceil-1$, Theorem~\ref{thm:tfdetdist} applies. To explicitly find $\rho(\mu(K_4))$, as in the previous example, note that in any 2-distinguishing coloring of $\mu (K_4)$ the ordered pairs of the form $(v_i,u_i)$ must be distinguished from each other. This forces us to use all four of the ordered pairs of two colors. This implies that we must use each of the colors on at least 4 vertices. Therefore,  $\rho (\mu (K_4))=4 = ((k+1) \lceil \log_2(k+1) \rceil)/2,$ precisely our upper bound on $\rho(\mu(G)).$

For many values of $t$ and $k=\det(G)$, it is possible to to find a coloring that does better in terms of the cost than the method above. For example, if $k+1$ is not a power of 2, we may choose a set of $k$ integers in the range $1,\dots,2^m$ that minimizes the number of 1s in their binary representations. When $t>m -1$, then choosing $k$ integers in the range $1,\dots,2^{t+1}$ gives a similar added flexibility. 

The next theorem shows this extra flexibility implies that when $t \ge \det(G) - 1$, the cost of 2-distinguishing $\gm$ achieves the lower bound of $\det(\gm) = \det(G).$

\begin{thm}\label{thm:twinfreeDet1}
 Let $G$ be a twin-free graph with no isolated vertices such that $\det(G) \geq 2$. Then for $t \geq \det(G) -1 $,
\[ 
\det(\gm) = \det(G), \dist(\gm) = 2, \text{ and } \rho(\gm) = \det(G).
\]
\end{thm}

\begin{proof}   
Since $t\ge \det(G)-1$ we have $t \geq \lceil \log_2(\det(G)+1)\rceil -1$. So by Theorem~\ref{thm:tfdetdist}, $\det(\gm) = \det(G)$ and $\dist(\gm) = 2$. 

Let $\det(G)=k$ and let$S = \{v_1,v_2,\dots,v_k\}$ be a determining set of $G$. Define $\widehat S =\{u_1^0, u_2^1, \dots, u_k^{k-1}\}$. Color the vertices in $\widehat S$ red and all other vertices blue. The proof that this is a 2-distinguishing coloring of $\gm$ is similar to the proof in Theorem~\ref{thm:tfdetdist}.

Now, since $|\widehat S| = k = \det(G)$, the size of a color class in the 2-distinguishing coloring above, $\rho(\mu_t(G)) \le \det(G).$ If there is a 2-distinguishing coloring of $\gm$ with a color class of size $\det(G)-1$, then $\gm$ would have a determining set of size $\det(G)-1$. Since $\det(\mu_t(G)) = \det(G)$, we can now conclude that $\rho(\mu_t(G))= \det(G)$ when $t\geq \det(G)-1$.\end{proof}

As noted earlier, Alikhani and Soltani~\cite{AS2018} showed that the classic Mycielski graphs $M_k = \mu^k(K_2)$ satisfy $\dist(M_n) = 2$ for any $ n \ge 2$. We can obtain this result and more by noting that $M_1 = C_5 \ne K_2$ is twin-free, has no isolated vertices, and satisfies $\det(M_1) = \det(C_5) = 2$.
We now apply Theorem~\ref{thm:twinfreeDet1} iteratively, with $t$ mercifully equal to $1$, to achieve the following.

\begin{cor}
For all $n \ge 2$, $\det(M_n) = \dist(M_n) = \rho(M_n) = 2$.
\end{cor}


\section{\boldmath ${\rm Det}(\mu_t(G))$ for $G$ with Twins} \label{sec:det&rhoTwins}

 We next consider graphs with twin vertices. For vertices $x$, $y$ of a graph $G$, define $x \sim y$ if $x$ and $y$ are twin vertices. It is easy to verify that $\sim$ is an equivalence relation on $V(G)$. 
 
 The quotient graph with respect to the relation $\sim$, denoted $\tG$, has as its vertices the set of equivalence classes $[x] = \{y \in V(G) \mid x \sim y\}$ with $[x]$ adjacent to $[z]$ in $\tG$ if and only if there exist $p \in [x]$ and $q \in [z]$ such that $p$ and $q$ are adjacent in $G$. By definition of $\sim$, all vertices in an equivalence class have the same neighbors, so in our case $[x]$ is adjacent to $[z]$ in $\tG$ if and only if $x$ is adjacent to $z$ in $G$. Thus,
 \[
 N_{\tG} ([x]) = \{ [z] \mid z \in N_G(x)\}.
 \]
 In particular, if
 $N_{\tG} ([x]) = N_{\tG} ([y])$, then $N_G(x) = N_G(y)$. Then $x$ and $y$ are twins in $G$, and so $[x]=[y]$ in $\tG$. This implies that $\tG$ is twin-free. In fact, $G$ is twin-free if and only if $\tG=G$. In this section we focus on graphs that have twins, that is, graphs for which $\tG \neq G$.
 
Given any set of mutually twin vertices in a graph $G$, there are automorphisms that permute the vertices in the set while fixing all other vertices. We may think of this as a type of local symmetry. By collapsing mutually twin vertices into a single vertex, the quotient graph $\tG$ captures the \lq non-twin' structure of $G$ and allows us to investigate the more global symmetry of $G$.
 
Since automorphisms preserve neighborhoods, and vertices of $G$ are identified in $\tG$ exactly when they have identical neighborhoods, every automorphism $\alpha$ of $G$ induces an automorphism $\widetilde \alpha$ of $\tG$ given by $\widetilde \alpha([x]) = [\alpha(x)]$. However, it need not be the case that all automorphisms of $\tG$ arise in this way. The only nontrivial automorphism of $G$ in Figure~\ref{fig:tGisP4} is the one interchanging the twin vertices $x$ and $y$, which induces the identity on $\tG = P_4$. However, $P_4$ has a nontrivial automorphism.

\begin{figure}[h]
  \centering
     \begin{tikzpicture}[scale=.7]
\draw[fill=black!100,line width=1] (0,0) circle (.2);
\draw[fill=black!100,line width=1] (135:2) circle (.2);
\draw[fill=black!100,line width=1] (-135:2) circle (.2);
\draw[fill=black!100,line width=1] (2,0) circle (.2);
\draw[fill=black!100,line width=1] (4,0) circle (.2);
\draw[black!100,line width=1.5pt] (0,0) -- (4,0);
\draw[black!100,line width=1.5pt] (0,0) -- (135:2);
\draw[black!100,line width=1.5pt] (0,0) -- (-135:2);
\draw (0,-.5) node{$u$};
\draw (2,-.5) node{$v$};
\draw (4,-.5) node{$w$};
\draw (-1.4,-2) node{$y$};
\draw (-1.4,1.9) node{$x$};
\draw (1,-1.5) node{$G$};
\begin{scope}[shift={(9,0)}]
\draw[fill=black!100,line width=1] (0,0) circle (.2);
\draw[fill=black!100,line width=1] (-2,0) circle (.2);
\draw[fill=black!100,line width=1] (2,0) circle (.2);
\draw[fill=black!100,line width=1] (4,0) circle (.2);
\draw[black!100,line width=1.5pt] (-2,0) -- (4,0);
\draw (0,-.5) node{$[u]$};
\draw (2,-.5) node{$[v]$};
\draw (4,-.5) node{$[w]$};
\draw (-2,-.5) node{$[x]$};
\draw (1,-1.5) node{$\tG$};
\end{scope}
\end{tikzpicture}
  \caption{A graph $G$ and its quotient graph $\tG$.} 
  \label{fig:tGisP4}
\end{figure}
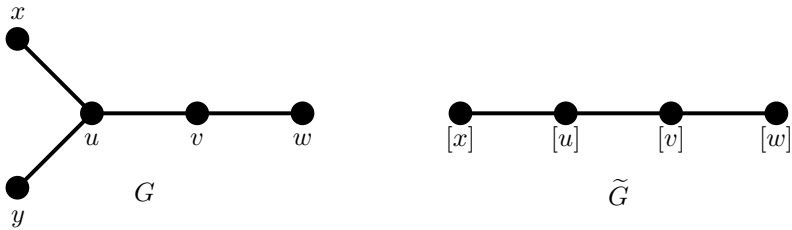

 Throughout the rest of this section, we will use $\tG$ to denote the quotient graph of $G$ and $\widetilde \alpha$ to denote an automorphism of its quotient graph. All sets of vertices with tilde notation will represent the corresponding sets of vertices in the quotient graph. 

We call a subset of $V(G)$ that contains all but one vertex from each equivalence class of twin vertices a {\em minimum twin cover} and denote it by $T$.  Recall that every determining set must contain all but one vertex from each collection of mutual twins. That is, every determining set must contain a minimum twin cover. Thus, if $T$ is a minimum twin cover of $G$, then $|T|\le \det(G)$. If a minimum twin cover is a determining set for $G$, then it is a minimum size determining set.

Given a minimum twin cover $T$, denote its image under the quotient map by $\tT$.  For example, for the graph $G$ in Figure~\ref{fig:tGisP4}, if $T = \{y\}$, then $\tT = \{[y]\}$. Furthermore, since $x$ and $y$ are twins, if $T = \{x\}$, we still have $\tT = \{[x]\} = \{[y]\}$. In fact, all minimum twin covers have the same image under the quotient map.  That is, for $T$ any minimum twin cover, $\widetilde T$ is the set of non-singleton equivalence classes.

\begin{lemma}\label{lem:autoInduced} Let $T$ be a minimum twin cover of $G$ and let $T \subseteq S$, for some $S\subseteq V(G)$. Suppose $\widetilde{\alpha}$ is an automorphism of $\tG$ that fixes $\tS$. Then there exists an automorphism $\alpha$ of $G$ that fixes $S$ and induces $\widetilde{\alpha}$.
\end{lemma}

\begin{proof} Since $\tT$ is precisely the set of non-singleton equivalence classes and $\widetilde \alpha$ fixes $\tT$, we can define 
\[
\alpha(x) = 
\begin{cases}
x, \quad &\text{ if } [x] \text{ is not a singleton},\\
y, &\text{ if } [x] \text{ is a singleton and } \widetilde \alpha([x]) = [y].
\end{cases}
\]
It is straightforward to check that $\alpha$ is an automorphism of $G$ that fixes $S$ and that $\widetilde \alpha ([v]) = [\alpha(v)]$ for all $v \in V(G)$. Thus, $\widetilde \alpha$ is induced by $\alpha$.
\end{proof}

\begin{cor}\label{cor:TtotT} Let $G$ be a graph.  If $S$ is a determining set for $G$, then $\tS$ is a determining set for $\tG$.\end{cor}

\begin{proof}
 Let $\widetilde{\alpha}$ be an automorphism of $\tG$ that fixes each element of $\tS$. Every determining set must contain a minimum twin cover and so by Lemma~\ref{lem:autoInduced}, there is an automorphism $\alpha$ of $G$ that fixes $S$ and induces $\widetilde \alpha$. Since $S$ is a determining set for $G$, by definition, $\alpha$ is the identity. Since $\widetilde\alpha$ is induced by $\alpha$, we get that $\widetilde \alpha$ is the identity.
\end{proof}

The converse of the preceding corollary does not hold if $\tS$ is a determining set for $\tG$ but does not contain $\tT$.  For example,  for the graph in Figure~\ref{fig:tGisP4}, $\tS = \{ [u]\}$ is a determining set for $\tG$, but $S = \{u\}$ is not a determining set for $G$ because it does not contain a minimum twin cover. However, if we restrict our attention to determining sets for $\tG$ that contain $\tT$, we can identify a set of \emph{residual} vertices, denoted by $R$ in the result below, to add to any twin cover of $G$ to obtain a determining set.

 \begin{thm}\label{thm:DettG}
 Let $T$ be a minimum twin cover of $G$ and let $\tS$ be a determining set for $\tG$ containing $\tT$. Let $R=\{x\in V(G) \ | \ [x]\in \tS\setminus \tT\}$. Then $S=T \cup R$ is a determining set for $G$. Furthermore, if $\tS$ is of minimum size among determining sets for $\tG$ that contain $\tT$, then $S$ is a minimum size determining set for $G$.

\end{thm}

\begin{proof}
Given $T$ and $\tS$, define $S$ as above. First we will show that $S$ is a determining set for $G$. Let $\alpha$ be an automorphism of $G$ that fixes each vertex in $S$. Let $\widetilde \alpha$ be the automorphism of $\tG$ induced by $\alpha$. Since $\alpha$ fixes each vertex in $S$, $\widetilde \alpha$ fixes each vertex in $\tS$. Since $\tS$ is a determining set for $\tG$, $\widetilde{\alpha}$ must be the identity on $\tG$. This implies that, in particular, $\widetilde{\alpha}$ fixes all singleton equivalence classes and so $\alpha$ fixes all twin-less vertices. Since $\widetilde{\alpha}$ also fixes non-singleton equivalence classes, $\alpha$ preserves equivalence classes of twins setwise. So, if a vertex $x$ has a twin, then either $x\in T\subseteq S$ and is fixed by $\alpha$, or $x\notin T$, but all of the twins of $x$ are in $T$, and are therefore fixed. In the latter case, since $\alpha$ preserves twins setwise, $x$ can only be mapped to one of its twins, and $x$ is fixed as well. Thus, $\alpha$ is the identity and $S$ is a determining set for $G$.

We will now show that if $\tS$ is a minimum size for a determining set for $\tG$ that contains $\tT$, then $S$ is a minimum size determining set for $G$. Suppose that $S$ is not of minimum size, and therefore, there is a determining set $B$ for $G$ such that $|B| < |S|$. We will show $|\tB|<|\tS|$ and that $\tB$ is a determining set for $G$, a contradiction. 

Since $B$ must contain a minimum twin cover, we may assume without loss of generality, by swapping a vertex with one of its twins if necessary, that $T \subseteq B$. By definition $T\subseteq S$ as well, so $|B\backslash T| < |S\backslash T|$. Because $B$ is of minimum size, no vertex in $B\backslash T$ has a twin in $V(G)$, and therefore the vertices in $\tB\backslash \tT$ are singleton equivalence classes. Thus, $ |\tB\backslash \tT|=|B\backslash T|$. The same is true for $S \backslash T$ by definition of $S$, and so $|\tS\backslash \tT|=|S\backslash T|$. All together we have that $|\tB\backslash \tT| < |\tS\backslash \tT|$. Since $\tB$ and $\tS$ each contain $\tT$, we now conclude that $|\tB| < |\tS|$. 

To show $\tB$ is a determining set for $\tG$, let $\widetilde \alpha$ be an automorphism of $\tG$ that fixes each element of $\tB$, and hence, each element of $\tT$. Let $\alpha$ be the automorphism of $G$  defined in  Lemma~\ref{lem:autoInduced} that induces $\widetilde \alpha$. Since $\widetilde{\alpha}$ fixes $\tB$ it fixes both $\tT$ and the singleton equivalence classes in $\tB\backslash \tT$. Because $\widetilde{\alpha}$ fixes the singleton equivalence classes in $\tB$, $\alpha$ fixes the twin-less vertices in $B$.  As proved in Lemma~\ref{lem:autoInduced}, $\alpha$ also fixes the vertices of $T$. Thus, $\alpha$ fixes all vertices in $B$. By the assumption that $B$ is a determining set for $G$, $\alpha$ is the identity on $G$. Since it is induced by $\alpha$, $\widetilde \alpha$ is the identity on $\tG$.  Thus $\tB$ is a determining set for $\tG$ of size smaller than $\tS$, a contradiction to our choice of $\tS$.\end{proof}

We note here that the proof of Theorem~\ref{thm:DettG} actually says something stronger: given a determining set $\tS$ for $\tG$ containing $\tT$, we can find $R$ so that for any minimum twin cover $T$, $S=T\cup R$ is a determining set for $G$.

Theorem~\ref{thm:DettG} yields natural bounds on $\det(G)$ in terms of $|T|$ and $\det(\tG)$. When the image $\tT$ of a minimum twin cover $T$ of $G$ yields a determining set for $\tG$, then $\det(G) = |T|$. Alternatively, when the image $\tT$ of a minimum twin cover of $G$ is disjoint from any minimum size determining set for $\tG$, then $\det(G) = |T| + \det(\tG)$.  Figures~\ref{fig:tGisP4} and~\ref{fig:C4Pend} depict both such scenarios. 

\begin{ex}\label{ex:lower}
For the graph $G$ in Figure~\ref{fig:tGisP4}, $T = \{y\}$ is both a minimum twin cover  and a minimum size determining set for $G$. Thus, $\det(G)=|T|$. Note that in this case $\tT = [x]$ yields a determining set for $\tG$ and $R$ is empty. 
Furthermore, by adding pendant edges to $G$ at $u$, increasing the number of twins of $x$ and the size of $T$, we can generate an infinite family of graphs for which $\det(G)=|T|$.
\end{ex}

\begin{ex}\label{ex:upper}
Let $G$ be the graph in Figure~\ref{fig:C4Pend} with $T= \{x_2, \dots x_n\}$. Here we can see that $\tG = P_5$ and $\tT$ consists only of its central vertex. This is not a determining set for $\tG$, although any singleton equivalence class of $V(\tG)$ would be. Any two-vertex set containing the central vertex, such as $\tS = \{[x_1], [v]\}$, is a minimum size determining set for $\tG$ that contains $\tT$. For this $\tS$, we get $R = \{v\}$. 
By Theorem~\ref{thm:DettG}, $S= \{x_2, \dots x_n, v\}$ is a minimum size determining set for $G$. This family of examples satisfies $\det(G) = |T| + \det(\tG)$.
\end{ex}

\begin{figure}[h]
  \centering
\begin{tikzpicture}[scale=.6]
\draw[fill=black!100,line width=1] (0,-2) circle (.2);
\draw[fill=black!100,line width=1] (0,2) circle (.2);
\draw[fill=black!100,line width=1] (2,0) circle (.2);
\draw[fill=black!100,line width=1] (-2,0) circle (.2);
\draw[fill=black!100,line width=1] (-4,0) circle (.2);
\draw[fill=black!100,line width=1] (4,0) circle (.2);
\draw[fill=black!100,line width=1] (0,1) circle (.2);
\draw[fill=black!100,line width=1] (0,0) circle (.1);
\draw[fill=black!100,line width=1] (0,-.5) circle (.1);
\draw[fill=black!100,line width=1] (0,-1) circle (.1);
\draw[black!100,line width=1.5pt] (0,-2) -- (2,0) -- (0,2) -- (-2,0) -- (0,-2);
\draw[black!100,line width=1.5pt] (-2,0) -- (0,1) -- (2,0);
\draw[black!100,line width=1.5pt] (-4,0) -- (-2,0);
\draw[black!100,line width=1.5pt] (4,0) -- (2,0);
\draw (-4,-.5) node{$u$};
\draw (-2,-.5) node{$v$};
\draw (2,-.5) node{$w$};
\draw (4,-.5) node{$z$};
\draw (0,1.4) node{$x_2$};
\draw (0,2.5) node{$x_1$};
\draw (0,-2.5) node{$x_n$};
\end{tikzpicture}
  \caption{A graph $G$ for which no minimum twin cover is a determining set.}
  \label{fig:C4Pend}
\end{figure}
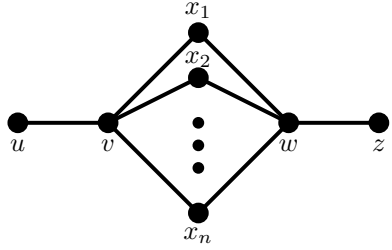

We note here that the set of residual vertices introduced in Theorem~\ref{thm:DettG} and denoted by $R$, consists solely of vertices without twins. Thus, the set $R=\{x\in V(G) \ | \ [x]\in \tS\setminus \tT\}$ has a 1 to 1 correspondence with the associated set $\tR =  \tS\setminus \tT$ in $\tG$.

We now formalize the bounds given by Theorem~\ref{thm:DettG} in the following corollary.

\begin{cor}\label{cor:twinBounds}
Let $T$ be a minimum twin cover of $G$. Then 
\[
|T| \leq \det(G) \leq |T| {+} \det(\tG),
\] with both bounds sharp.
\end{cor}

\begin{proof}
Let $\widetilde D$ be a minimum size determining set for $\tG$. Then $\det(\tG) = |\widetilde D|$ and $\tS = \tT \cup \widetilde D$ is a determining set for $\tG$ containing the set of non-singleton equivalence classes $\tT$. Note that $\widetilde D$ and $\tT$ are not necessarily disjoint.  However, we can write $\tS$ as the disjoint union of $\tT$ and $\tR=\widetilde D\setminus \tT$ and note that $0\leq |\tR|\leq |\widetilde D|=\det(\tG)$. Now by 
Theorem~\ref{thm:DettG}, $S=T\cup R$ is a minimum size determining set for $G$ of size  $|T| +|R|$.  By the bounds on $|\tR|$ found above, and the fact that $|\tR|=|R|$, $|T| \leq |S| \leq |T|+\det(\tG)$.

Example~\ref{ex:lower} gives sharpness in the lower bound while Example~\ref{ex:upper} gives sharpness in the upper bound.\end{proof}

The preceding corollary can be extended to establish bounds on $\det(\gm)$ in the case where $G$ has twins. To do so, we must investigate how applying the generalized Mycielski construction affects the size of a minimum twin cover, as well as the relationship between $\det(\tG)$ and $\det(\widetilde{\gm})$.

\begin{lemma}\label{lem:twinrepsmu} Suppose that $G$ has no isolated vertices and let $T$ be a minimum twin cover of $G$. Then for $t\geq 1$, the set consisting of vertices in $T$ and all of their shadows, 
\[
\Tt = \{u_i^s \mid v_i \in T, \, 0 \le s \le t \},
\]
is a minimum twin cover of $\gm$ of size $(t{+}1)|T|$.\end{lemma}

\begin{proof}
By Observations~\ref{obs:twins1'},~\ref{obs:twins2'}, and~\ref{obs:twins3'}, twin vertices in $\gm$ must be shadows of twins of $G$, and must reside at the same level in $\gm$. Thus, if $T$ contains all but one vertex from any set of mutual twin vertices in $G$, then the copy of $T$ at level $s$ contains all but one vertex from each set of mutual twins in $\gm$ at level $s$. So $T_t$ is a minimum twin cover of $\gm$.\end{proof}

The following lemma proves that the processes of applying the generalized Mycielski construction commutes with the process of applying the canonical map onto the quotient graph. Figure~\ref{fig:quotientK13} illustrates this for the graphs $ K_{1,3}$ and $\mu_t(K_{1,3})$ with $t = 1, 2$, and their quotient graphs. 

\begin{lemma}\label{lem:commutes}
If $G$ has no isolated vertices, then for $t\geq 1$, $\mu_t(\tG) = \widetilde{\gm} $. 
\end{lemma}
\begin{proof}
By Observations~\ref{obs:twins1'},~\ref{obs:twins2'}, and~\ref{obs:twins3'},
two vertices are twins in $\gm$ if and only if either they are twins vertices in $G$ or they are shadows at a given level $s$ of twin vertices in $G$. In terms of our equivalence relation,
$v_i \sim v_j $ in $ G$ if and only if $u_i^s \sim u_j^s $ for all $0 \le s \le t $ in $ \gm.
$
This allows us to map the shadows at level $s$ of $[v_i]$ in $\mu_t(\tG)$ to $[u_i^s]$ in $\widetilde{\gm}$. Additionally, if $w$ is the root in $\gm$, then we map the root of $\mu_t(\tG)$ to $[w]$ in $\widetilde{\gm}$. It it straightforward to verify that this map preserves both adjacencies and non-adjacencies. Therefore, $\mu_t(\tG) = \widetilde{\gm}$.\end{proof}

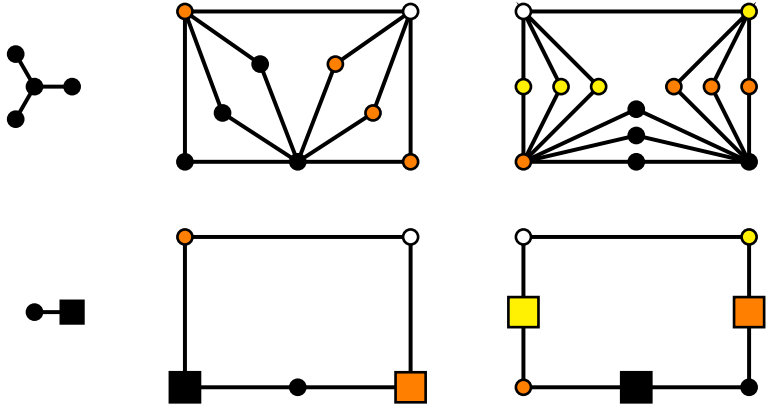
\begin{figure}[h]
  \centering
   \begin{tikzpicture}[scale=.5]
\draw[fill=black!100,line width=1] (0,0) circle (.2);
\foreach \j in {1,...,3}
{\draw[fill=black!100,line width=1] (120*\j:1) circle (.2);\draw[black!100,line width=1.5pt] (120*\j:1) -- (0,0);
}
\begin{scope}[shift={(0,-6)}]
\draw[fill=black!100,line width=1] (0,0) circle (.2);
\draw[black!100,line width=1.5pt] (1,0) -- (0,0);
\draw[fill=black!100,line width=1] (1.3,.3) rectangle (.7,-.3);
\end{scope}
\begin{scope}[shift={(7,-2)}]
\draw[black!100,line width=1.5pt] (-3,0) -- (3,0); 
\draw[black!100,line width=1.5pt] (-3,4) -- (3,4); 
\draw[black!100,line width=1.5pt] (-3,0) -- (-3,4); 
\draw[black!100,line width=1.5pt] (3,0) -- (3,4); 
\draw[black!100,line width=1.5pt] (0,0) -- (2,1.3); 
\draw[black!100,line width=1.5pt] (0,0) -- (1,2.6); 
\draw[black!100,line width=1.5pt] (3,4) -- (2,1.3); 
\draw[black!100,line width=1.5pt] (3,4) -- (1,2.6);
\draw[black!100,line width=1.5pt] (0,0) -- (-2,1.3); 
\draw[black!100,line width=1.5pt] (0,0) -- (-1,2.6); 
\draw[black!100,line width=1.5pt] (-3,4) -- (-2,1.3);
\draw[black!100,line width=1.5pt] (-3,4) -- (-1,2.6);
\draw[fill=black!100,line width=1] (0,0) circle (.2);
\draw[fill=black!100,line width=1] (-3,0) circle (.2);
\draw[fill=orange!100,line width=1] (3,0) circle (.2);
\draw[fill=white!100,line width=1] (3,4) circle (.2);
\draw[fill=orange!100,line width=1] (-3,4) circle (.2);
\draw[fill=orange!100,line width=1] (2,1.3) circle (.2);
\draw[fill=orange!100,line width=1] (1,2.6) circle (.2);
\draw[fill=black!100,line width=1] (-2,1.3) circle (.2);
\draw[fill=black!100,line width=1] (-1,2.6) circle (.2);
\end{scope}
\begin{scope}[shift={(7,-8)}]
\draw[black!100,line width=1.5pt] (-3,0) -- (3,0); 
\draw[black!100,line width=1.5pt] (-3,4) -- (3,4); 
\draw[black!100,line width=1.5pt] (-3,0) -- (-3,4); 
\draw[black!100,line width=1.5pt] (3,0) -- (3,4); 
\draw[fill=black!100,line width=1] (0,0) circle (.2);
\draw[fill=orange!100,line width=1] (-3,4) circle (.2);
\draw[fill=black!100,line width=1] (-2.6,.4) rectangle (-3.4,-.4);
\draw[fill=orange!100,line width=1] (3.4,.4) rectangle (2.6,-.4);
\draw[fill=white!100,line width=1] (3,4) circle (.2);
\end{scope}
\begin{scope}[shift={(16,-6)}]
\draw[black!100,line width=1.5pt] (-3,8) -- (3,8); 
\draw[black!100,line width=1.5pt] (-3,8) -- (-3,4) -- (3,4);
\draw[black!100,line width=1.5pt] (3,4) -- (0,4.7) -- (-3,4); 
\draw[black!100,line width=1.5pt] (-3,4) -- (0,5.4) -- (3,4) -- (3,8); 
\draw[black!100,line width=1.5pt] (-3,4) -- (-2,6) -- (-3,8) -- (-1,6) -- (-3,4);
\draw[black!100,line width=1.5pt] (3,4) -- (2,6) -- (3,8) -- (1,6) -- (3,4);
\draw[fill=black!100,line width=1] (3,4) circle (.2);
\draw[fill=black!100,line width=1] (0,4) circle (.2);
\draw[fill=black!100,line width=1] (0,4.7) circle (.2);
\draw[fill=black!100,line width=1] (0,5.4) circle (.2);
\draw[fill=orange!100,line width=1] (-3,4) circle (.2);
\draw[fill=white!100,line width=1] (-3,8) circle (.2);
\draw[fill=yellow!100,line width=1] (3,8) circle (.2);
\draw[fill=yellow!100,line width=1] (-3,6) circle (.2);
\draw[fill=orange!100,line width=1] (3,6) circle (.2);
\draw[fill=orange!100,line width=1] (2,6) circle (.2);
\draw[fill=orange!100,line width=1] (1,6) circle (.2);
\draw[fill=yellow!100,line width=1] (-2,6) circle (.2);
\draw[fill=yellow!100,line width=1] (-1,6) circle (.2);
\end{scope}

\begin{scope}[shift={(16,-12)}]
\draw[black!100,line width=1.5pt] (-3,8) -- (3,8); 
\draw[black!100,line width=1.5pt] (3,4) -- (3,8);
\draw[black!100,line width=1.5pt] (-3,8) -- (-3,4) -- (3,4);
\draw[fill=black!100,line width=1] (3,4) circle (.2);
\draw[fill=black!100,line width=1] (.4,4.4) rectangle (-.4,3.6);
\draw[fill=orange!100,line width=1] (-3,4) circle (.2);
\draw[fill=white!100,line width=1] (-3,8) circle (.2);
\draw[fill=yellow!100,line width=1] (3,8) circle (.2);
\draw[fill=yellow!100,line width=1] (-2.6,6.4) rectangle (-3.4,5.6);
\draw[fill=orange!100,line width=1] (3.4,6.4) rectangle (2.6,5.6);
\end{scope}
\end{tikzpicture}
 
  \caption{The graphs $ K_{1,3}$ and $\mu_t(K_{1,3})$ for $t = 1, 2$, and their quotient graphs, with the equivalence classes of twin vertices shown as squares.} \label{fig:quotientK13}
\end{figure}

In light of Lemma~\ref{lem:commutes}, for the remainder of the paper, we will use the notation $\mu_t(\widetilde{G})$ rather than $\widetilde{\mu_t(G)}$.

In the following lemma, we use Theorem~\ref{thm:twinfreeDet} and
 Lemma~\ref{lem:commutes} to show that the quotient graphs of $G$ and $\gm$ have the same determining number when $G\ne K_{\ell, m}$.

\begin{lemma}\label{lem:EqualDet}
If $G \ne K_{\ell, m}$ for any $\ell,m \ge 1$ and $G$ has no isolated vertices, then for $t \geq 1$, \[\det(\tG)= \det(\mu_t(\widetilde{G})).\]
\end{lemma}

\begin{proof} The assumption that $G \ne K_{\ell, m}$ implies that $\tG \ne K_2$. Moreover, $\tG$ is twin-free with no isolated vertices and so by Theorem~\ref{thm:twinfreeDet}(\ref{thm:twinfreeDet:notK2}), $ \det(\mu_t(\tG)) = \det(\tG)$.\end{proof}

We now have the tools to specify precisely how the presence of twins affects the determining number of the generalized Mycielskian of a graph.

 \begin{thm}\label{thm:twinDet}
 Let $G$ be a graph with twins and no isolated vertices. Let $T$ be a minimum twin cover of $G$. Then for $t \geq 1$,

 \[\det(\gm)=t|T| + \det(G).\]
\end{thm}
\begin{proof}

We first consider the case where $G = K_{\ell, m}$ for $1 \leq \ell \leq m$. Since $G$ has twin vertices, it follows that $m \geq 2$. Let $ \{v_1, \dots, v_\ell\}$ and $\{ v_{\ell+1}, \dots v_{\ell + m}\}$ be the partite sets in $V(K_{\ell,m})$. Since every minimum twin cover contains all but one vertex from each partite set, without loss of generality, we can choose $T = V(K_{\ell,m}) \setminus \{v_\ell, v_{\ell+m}\}$ as a minimum twin cover of $G$. It is straightforward to verify that $T$ is a minimum size determining set for $K_{\ell, m}$, and therefore, that $|T|=\det(G)$. 

By Lemma~\ref{lem:twinrepsmu}, $\Tt$ is a minimum twin cover of $\mu_t(G)$ of size $(t+1)|T|$, which in this case is equal to $t|T|+\det(G)$. We next show that $\Tt$ is a minimum size determining set for $\mu_t(G)$. 

Let $\widehat \alpha$ be an automorphism of $\mu_t(G)$ that fixes the vertices in $\Tt$. Since $\ell \ge 1$ and $m \ge 2$, all vertices $u_i^j$ for $\ell+1 \le i \le \ell+m$ and $0 \le j \le t$ are in non-singleton equivalence classes. Since $\widehat \alpha$ fixes all but one vertex in a non-singleton equivalence class of twins, it must then fix the remaining vertex as well. Thus, since we have chosen $\widehat \alpha$ to fix $\Tt$, it fixes each vertex that has a twin. The remaining vertices to be considered are the root $w$, and if $\ell=1$, the vertices $u_1^0,\dots,u_1^t$.  The latter are the center of the star $K_{1,m}$ and its shadows. Continuing our examination of $K_{1,m}$, since for $0\leq i\leq t-1$, $u_1^i$ is only adjacent to shadow vertices of the other partite set, each of $u_1^0,\dots,u_1^{t-1}$ has a unique neighborhood in $\gm$ that is entirely fixed by $\widehat \alpha$. Thus, each of these vertices must also be fixed by $\widehat \alpha$. Further, the vertex $u_1^t$ is adjacent to vertices $u_{2}^{t-1},\dots,u_{m+1}^{t-1}$, all of which are fixed by $\widehat \alpha$, while the vertex $w$ is adjacent to none of these. Therefore, $\widehat \alpha$ must also fix $u_1^t$ and $w$, and so it fixes all vertices of $\gm$. Consequently, in all cases, the minimum twin cover $\Tt$ is a determining set for $\mu_t(G)$, and so is a minimum size determining set.

Now suppose $G\neq K_{\ell,m}$. It may no longer be the case that the minimum twin cover $T$ is a determining set for $G$.  However, we can invoke Theorem~\ref{thm:DettG} to find one. Among all determining sets for $\tG$ that contain $\tT$, let $\tS$ be one of minimum size. Define $R =  \{x \in V(G) \mid [x]\in \tS \setminus \tT\}$ as in Theorem~\ref{thm:DettG} and it follows that $\widetilde R = \tS \setminus \tT$. By Theorem~\ref{thm:DettG}, the disjoint union $T\cup R = S$ is a minimum size determining set for $G$ and, therefore, $\det(G) = |T| + |R|.$

Define $D = \Tt \cup R$, where $\Tt$ is a minimum twin cover for $\gm$, with $T_t$ defined as in Lemma~\ref{lem:twinrepsmu}. We will show $D$ is a minimum size determining set for $\gm$ by first showing $\tD = \widetilde{T_t} \cup \tR$ is a determining set for $\mu_t(\tG)$. 

By definition, $\Tt = \{u_i^s \mid v_i \in T, \, 0 \le s \le t \}$ contains $T$ and so $\widetilde{T_t}$ contains $\tT$. Thus, $\tD = \widetilde{T_t} \cup \tR$ contains  $\tS = \tT \cup \tR$. By definition, $\tS =\tT\cup \tR $ is a determining set for $\tG$. By Theorem~\ref{thm:twinfreeDet}, since $\tG$ is twin-free and $\tG \neq K_2$, we find that $\tS$ is a  determining set for $\mu_t(\tG)$. Since $\tD$ contains $\tS$, it follows that $\tD$ is a determining set for $\mu_t(\tG)$.

We now apply Theorem~\ref{thm:DettG} to $\mu_t(\tG)=\widetilde{\gm}$. By the above arguments, $T_t$ is a minimum twin cover of $\gm$ and $\tD$ is a determining set for $\mu_t(\tG)$ containing $\widetilde{T_t}$. Moreover, $\tD = \widetilde{T_t} \cup \tR$ and it follows immediately from Theorem~\ref{thm:DettG} that $D = T_t \cup R$ is a determining set for $\gm$. 

To show $D$ is of minimum size, suppose to the contrary that there exists a set $B$ such that $|B| < |D|$ and $B$ is a determining set for $\gm$. Similar to the proof of Theorem~\ref{thm:DettG}, we will show that this implies that there exists a determining set for $\mu_t(\tG)$ of size less than $|\tS|$, a contradiction.

Since $B$ is a determining set for $\gm$, $B$ must contain a minimum twin cover of $\mu_t(G)$. We can assume without loss of generality, by replacing some twins in the cover with different twins if necessary, that $B$ contains the twin cover $T_t$. Let $J=B\setminus T_t$. Since $|B| < |D|$, it follows that $|J| < |R|$. 

By assumption, $B$ is a determining set for $\gm$ that contains the  minimum twin cover $T_t$ and so by Corollary~\ref{cor:TtotT}, $\tB = \widetilde J \cup \widetilde{\Tt}$ is a determining set for $\mu_t(\tG)$.

Let $\widetilde{B'}$ be the projection of $\tB$ onto the set of original vertices in $\mu_t(\tG)$. More formally, $\widetilde{B'} = \{ [u_i^0] \mid [u_i^j] \in \tB \text{ for some  } 0 \leq j \leq t \}$. Observe that $|\widetilde{B'}| \leq |\widetilde J \cup \tT| < |\tS|$. Let $\widetilde{\alpha}$ be an automorphism of $\mu_t(\tG)$ that fixes $\widetilde{B'}$. Since $\mu_t(\tG)$ is twin-free, we can apply Lemma~\ref{lem:GenA&Sfixed}(\ref{lem:GenA&Sfixed:shadows}) to conclude that if vertex $[u_i^0]$ is fixed by $\widetilde \alpha$, vertices $[u_i^j]$ for $1 \leq j \leq t$ are also fixed by $\widetilde \alpha$. It follows that $\widetilde \alpha$ fixes $\tB$ and so it is the trivial automorphism. Thus, $\widetilde{B'}$ is a determining set in $\mu_t(\tG)$ containing $\tT$ of size less than $\tS$, a contradiction.

We have shown that $D = T_t \cup R$ is a minimum size determining set for $\gm$. 
By Lemma~\ref{lem:twinrepsmu}, $|D| = |T_t| + |R| = (t+1)|T| + |R|$. 
With a touch of algebra and the fact that $\det(G) = |T|+|R|$, we have \vskip.1in

\hspace{1.25in} $\det(\gm) = t|T| + \det(G)$.\end{proof}

We now combine Theorems~\ref{thm:twinfreeDet} and \ref{thm:twinDet} to state a result that applies to all graphs.

 \begin{thm}\label{thm:GrandFinale}
 Let $G$ be a graph with no isolated vertices and let $T$ be a (possibly empty) minimum twin cover of $G$. Then for $t \geq 1$,
 
 \begin{enumerate}[(i)]
  
   \item If $G = K_2$ then, $\det(G) = 1$ and $\det(\gm) = 2.$ 

    \item If $G \neq K_2$ then, $\det(\gm) = t|T| + \det(G).$
    \end{enumerate}

\end{thm}

Further, we can iterate this result to see what happens when applying the generalized Mycielskian construction numerous times
 We let $\mu_t^1(G) = \gm$ and  for $k \ge 1$, we recursively define $\mu_t^{k+1}(G) = \mu_t (\mu_t^k(G))$.
 
\begin{cor}
Let $G$ be a graph with no isolated vertices and let $T$ be a (possibly empty) minimum twin cover of $G$. Then for $t, k \ge 1$
\begin{enumerate}[(i)]
\item If $G = K_2$ then, $\det(G) = 1$ and $\det(\mu_t^k(G)) = 2.$ 

    \item If $G \neq K_2$, then 
    $\det(\mu_t^k(G)) = [(t+1)^k -1]|T| + \det(G).$
\end{enumerate}
\end{cor}
\begin{proof}
Both parts can be proved by induction on $k$. For (i), note that $\mu^k_t(K_2)$ is twin-free and so by  Theorem~\ref{thm:GrandFinale}(ii),  $\det(\mu_t^{k+1}(K_2)) = \det(\mu^k_t(K_2)) = 2$. For (ii), note that by Lemma~\ref{lem:twinrepsmu}, each application of the generalized Mycielski construction increases the size of a minimum twin cover by a factor of $(t+1)$.
\end{proof}

In particular,  for $t = 1$, when $G=K_2$, we obtain the determining number of the classic Mycielski graphs, 
$\det(\mu^k(K_2)) =\det(M_k) = 2$. For $G \neq K_2$,
$\det(\mu^k(G)) = (2^k - 1) |T| + \det(G).$
This shows that iterating the regular Mycielskian construction on a graph with twins causes the determining number to grow faster than the non-iterated, but multi-layer generalized Mycielskian of the same graph. Put another way, the determining number grows linearly with $t$ but exponentially with $k$.

\section{Acknowledgments}

The work in this article is a result of a collaboration made possible by the Institute for Mathematics and its Applications' Workshop for Women in Graph Theory and Applications, August 2019.


    

\FloatBarrier
\bibliographystyle{plain}
\bibliography{DetResubmission}

\end{document}